\documentclass[11pt, a4paper]{amsart}
\usepackage[utf8]{inputenc}
\usepackage[T1]{fontenc}
\usepackage[margin=0.9in]{geometry}

\usepackage{comment}
\usepackage{lmodern}
\usepackage{upgreek}
\usepackage{color}
\usepackage{amsmath}
\usepackage{amssymb}
\usepackage{amsthm}
\usepackage{extarrows}

\usepackage{pdfpages}

\usepackage{fancyhdr}

\usepackage{setspace}

\newcommand{\m}{\mathfrak m}

\newcommand{\depth}{\operatorname{depth}}

\usepackage{hyperref}
\hypersetup{
	colorlinks=true,
	linkcolor=blue,
	urlcolor=blue,
	citecolor=blue,
	pdftitle={On an Instance of   the Small Cohen-Macaulay Conjecture II},
	pdfauthor={Likun Xie}
}

\numberwithin{equation}{section}
\newtheorem{theorem}{Theorem}[section]
\newtheorem{lemma}[theorem]{Lemma}
\newtheorem{proposition}[theorem]{Proposition}
\newtheorem{corollary}[theorem]{Corollary}

\theoremstyle{remark}
\newtheorem{remark}[theorem]{Remark}

\newtheorem{example}{Example}[section]

\theoremstyle{definition}

\newtheorem*{prob*}{Problem}

	\title[On an Instance of   the Small Cohen-Macaulay Conjecture II]
{On an Instance of   the Small Cohen-Macaulay Conjecture II}
\author{Likun Xie}
 
\address{Max-Planck-Institut für Mathematik
	Vivatsgasse 7, 53111, Bonn, Germany} 
\email{xie@mpim-bonn.mpg.de}

\begin{document}
	
\begin{abstract}
We show that any $d$-dimensional local ring $A$ with a dualizing
complex, $\depth A=d-1$,  and cyclic deficiency module
\(K^{d-1}(A)\) admits a maximal Cohen--Macaulay module.  It is constructed as the unique nonzero cohomology module of the cone of the derived morphism induced by a
	surjection \(A\to K^{d-1}(A)\). This gives new cases of the small Cohen--Macaulay conjecture. When \(A\) is quasi-Gorenstein, this
	module is identified with the first syzygy of the canonical module
	\(\omega_{A/xA}\), for any
	\(x\in\operatorname{ann}_A K^{d-1}(A)\) that is regular on \(A\).
	This recovers a theorem of Tavanfar and Shimomoto for quasi-Gorenstein  deformations of a 3-dimensional quasi-Gorenstein local ring  $A$ with \(K^2(A)\cong k\).
	We also construct examples of section rings satisfying the hypotheses of our theorem.
\end{abstract}

	\maketitle	
\section{Introduction}	

The existence of maximal Cohen--Macaulay modules over  local
rings is a fundamental problem in commutative algebra. While such modules arise naturally over Cohen--Macaulay rings, their existence over non-Cohen--Macaulay rings is far less understood, and relatively few general constructions are available. 
The small Cohen–Macaulay conjecture states that every complete local ring admits a maximal Cohen–Macaulay module, that is, a finitely generated Cohen–Macaulay module of maximal dimension \cite[p.~10]{Hochster1975}.

Very recently, Anghel has announced counterexamples to the graded version of the small Cohen–Macaulay conjecture in equal characteristic zero \cite{anghel}, and Chen has proposed a counterexample to the small Cohen–Macaulay conjecture in equal characteristic zero \cite{chen}. The conjecture  remains open in general for complete local domains of positive residue characteristic.  
The significance of the small Cohen--Macaulay conjecture extends beyond the statement itself.   At a more fundamental level lies the much broader and still less understood problem of understanding the structural conditions for the existence and nonexistence of maximal Cohen--Macaulay modules, as well as the rich structural and homological phenomena that the study of this conjecture may uncover.

A result of Tavanfar and Shimomoto \cite[Thm.~3.2]{tavanfar_Shimomoto}
gives one such existence class whose    construction is essentially 3-dimensional. Their theorem is stated in arbitrary
dimension, while the higher-dimensional construction is obtained by deformation from
a 3-dimensional situation.  More precisely, they consider a
quasi-Gorenstein local ring \((A,\mathfrak{ m},k)\)  which is a homomorphic image of a
Gorenstein local ring, and assume the existence of a regular sequence
\(\mathbf y\) such that \(A/\mathbf yA\) is a 3-dimensional
quasi-Gorenstein   local ring satisfying \(   H^2_{\mathfrak m}(A/\mathbf yA) =k.\) 
Under these hypotheses, there exists
$
x\in \operatorname{ann}_A K^{d-1}(A),
$ such that $x,\mathbf{y}$ form a regular sequence of $A$ and
the first syzygy \(\Omega^1\omega_{A/xA}\) of the canonical module   of
\(A/xA\) is maximal Cohen--Macaulay over $A$. This module is obtained by lifting
the maximal Cohen--Macaulay module
$
\Omega^1\omega_{A/(x,\mathbf y)A}
$
over the 3-dimensional quotient \(A/\mathbf yA\); the deformation
conditions ensure that the relevant deficiency module  and syzygy lift.
 
 Therefore,  the   proof and construction of Tavanfar and Shimomoto  reduce to the following 3-dimensional statement: if \((A,\m,k)\) is a 3-dimensional quasi-Gorenstein local ring with
\(\depth A=2\) and \(H^2_{\mathfrak m}(A)\cong k\), then  for any
nonzerodivisor \(x\in \m \)  the first syzygy
\(\Omega^1\omega_{A/xA}\) of the canonical module of \(A/xA\) is maximal
Cohen--Macaulay. In our previous work \cite{xie}, we gave a simpler proof of both the reduction step and the 3-dimensional case.
The proof showed that   \(H^2_{\mathfrak m}(A)\cong k\) enters through the structure of the deficiency module  \(K^{2}(A)\) and
naturally led to the question of what structural conditions on \(K^{d-1}(A)\)
suffice. We show here that cyclicity is sufficient in any dimension
and without assuming that \(A\) is quasi-Gorenstein. This gives new cases of the small Cohen--Macaulay conjecture.  

We now state our main result.
\begin{theorem}\label{main}
	Let $(A,\mathfrak m)$ be a local ring of dimension $d$ admitting a
	dualizing complex, or equivalently, a homomorphic image of a Gorenstein
	ring. Suppose that
	$
	\depth A=d-1
	$
	and that $K^{d-1}(A)$ is cyclic. Then $A$ admits a  maximal Cohen--Macaulay module.
\end{theorem}

The construction  is naturally formulated in the derived category. Let $D_A$ be a dualizing complex normalized so that $H^{-i}(D_A)=K^i(A)$, and set $X=D_A[d-1]$. Then $H^{-1}(X)=\omega_A$ and $H^0(X)=K^{d-1}(A)$. A generator of the cyclic module $K^{d-1}(A)$ determines a morphism $f\colon A\to X$ in $D(A)$, and we consider the module
$
M:=H^{-1}\bigl(\operatorname{Cone}(f)\bigr).
$ 
The surjectivity of the induced map $A\to K^{d-1}(A)$ implies that the cone has cohomology concentrated in degree \(-1\). Using duality, we show that the dual of this cone is again, up to shift,  the cone associated to another generator of \(K^{d-1}(A)\), and hence also has cohomology concentrated in
degree \(-1\).  Local duality then implies   \(M\) is maximal Cohen--Macaulay.

Although the proof can be presented succinctly in the derived setting, we
also give a concrete description of the resulting module \(M\). In the setup of Theorem \ref{main}, suppose that
\(x\in \operatorname{ann}_A K^{d-1}(A)\) is a nonzerodivisor on \(A\).
We show in Lemma \ref{cone_identification} that \(M\) is the kernel of a natural map
\[
\Phi\colon \omega_A\oplus A \longrightarrow \omega_{A/xA}.
\]
Moreover, when \(A\) is quasi-Gorenstein, namely when
\(\omega_A\cong A\), we show that \(M\) is precisely the first syzygy
\(\Omega^1\omega_{A/xA}\) of the canonical module of \(A/xA\), see Proposition \ref{cor_syzygy}.
Thus our construction recovers  the 3-dimensional construction of Tavanfar
and Shimomoto.
As pointed out by the referee, in the quasi-Gorenstein setting with
\(\depth A=\dim A-1\), the  deformation condition of Tavanfar
and Shimomoto can be reformulated
as \(e_{\mathfrak m}(K^{d-1}(A))=1\), which in this case implies that
\(K^{d-1}(A)\) is cyclic; see Proposition~\ref{equivalent}.
Thus our construction also applies directly to their higher-dimensional
deformation ring without passing to the 3-dimensional quotient, and Proposition \ref{cor_syzygy}
gives \(\Omega^1\omega_{A/xA}\) as a maximal Cohen--Macaulay module for
every nonzerodivisor
\(x\in\operatorname{ann}_A K^{d-1}(A)\), independently of the deformation sequence.

In this setting,  \(\omega_{A/xA}\) is a Cohen--Macaulay module of dimension \(d-1\).
This fits naturally into a general question about when higher syzygies
of lower-dimensional Cohen--Macaulay modules are maximal Cohen--Macaulay. In Section~\ref{higher_syzygy}, we briefly discuss a condition for the
\(r\)-th syzygy of a \(d-r\)-dimensional Cohen--Macaulay module to be
maximal Cohen--Macaulay.

Finally, we give examples of families of section rings satisfying the
hypotheses of Theorem~\ref{main}, including some rings whose completions
admit no small Cohen--Macaulay algebra.
 
\subsection*{Notation}

Throughout the paper, ``local ring'' means Noetherian local ring. We use the standard convention for cohomologically indexed complexes as in  Weibel \cite{weibel}: for a complex $C$, its $p$-th translate $C[p]$ is defined by
$
C[p]^n = C^{n-p},
$
with differential
$
d_{C[p]}^n = (-1)^p d_C^{\,n-p}.
$ Let $D(A)$ denote the derived category of $A$-modules, and let
$D_c^b(A)$ denote the full subcategory of $D(A)$ consisting of complexes
with bounded, finitely generated cohomology.

By \cite[Cor.~1.4]{dualizing}, a local ring admits a dualizing complex if
and only if it is a homomorphic image of a Gorenstein local ring; we use
these equivalent conditions interchangeably. If \(A\) is a homomorphic
image of a Gorenstein local ring \(R\) of dimension \(n\), then for a
finitely generated \(A\)-module \(M\) we write
$
K^i(M):=\operatorname{Ext}_R^{\,n-i}(M,R)
$
for the \(i\)-th deficiency module of \(M\).
 
For a local ring \((A,\mathfrak m,k)\), we denote by \((-)^\vee\) the
Matlis dual functor
$
(-)^\vee:=\operatorname{Hom}_A(-,E_A(k)),
$
where \(E_A(k)\) is the injective hull of \(k\) over \(A\). For a
\(k\)-vector space \(V\), we write \(V'\) for its \(k\)-linear dual.

\section{The Cone Construction}
\begin{proof}[Proof of Theorem \ref{main}]
Let $D_A$ denote a dualizing complex for $A$, normalized so that
$
H^{-i}(D_A)=K^i(A);
$
see \cite[Sec.~3]{canonical_module}. Set
\[
X:=D_A[d-1].
\]
Let \(\omega_A\) denote the canonical module of \(A\). Since \(\depth A=d-1\),   \(K^i(A)=0\) for
\(i\notin\{d-1,d\}\), therefore 
\[
H^{-1}(X)=\omega_A,\qquad
H^0(X)=K^{d-1}(A),\qquad
H^i(X)=0 \quad \text{for } i\notin\{-1,0\}.
\]
Denote \(C:=K^{d-1}(A)\). Since
\[
\operatorname{Hom}_{D(A)}(A,X)
=
H^0 (\mathbf{R}\!\operatorname{Hom}_A(A,X) )
\cong H^0(X)
=
C,
\]
a generator \(c\in C\) determines a morphism
$
f\colon A\to  X
$
in \(D(A)\). The induced map on $H^0$,
$
H^0(f)\colon A \to C,
$
is given by multiplication by \(c\) which  is therefore surjective.   There is an exact triangle
\begin{equation}\label{exact_triangle}
	A   \xlongrightarrow{f}X\longrightarrow \operatorname{Cone}(f)\longrightarrow A[-1].
\end{equation}
The associated long exact sequence in cohomology contains the exact sequence
\[
0\longrightarrow \omega_A
\longrightarrow H^{-1}(\operatorname{Cone}(f))
\longrightarrow A
\xrightarrow{H^0(f)} C
\longrightarrow H^0(\operatorname{Cone}(f))
\longrightarrow 0.
\]
Since \(H^0(f)\colon A\to C\) is surjective, we have \(H^0(\operatorname{Cone}(f))=0\). Hence \(\operatorname{Cone}(f)\) has cohomology concentrated in degree \(-1\), so
\[
\operatorname{Cone}(f)\simeq M[-1],
\qquad
M:=H^{-1}(\operatorname{Cone}(f)).
\]
Moreover,
the above long exact sequence yields a short exact sequence
\[
0\longrightarrow \omega_A
\longrightarrow M
\longrightarrow \operatorname{Ann}_A(C)
\longrightarrow 0.
\]
Define
\[
\mathbb{D}(-):=\mathbf{R}\!\operatorname{Hom}_A(-,D_A)[d-1].
\]
By \cite[Prop.~V.2.1]{duality} or \cite[\href{https://stacks.math.columbia.edu/tag/0A7C}{Tag 0A7C}]{stacks-project}, the functor $\mathbb{D}$ is a
contravariant equivalence on $D_c^b(A)$ and satisfies biduality
$
\mathbb{D}^2\simeq \operatorname{id}.
$
By construction,
\[
\mathbb{D}(A)=X
\qquad\text{and}\qquad
\mathbb{D}(X)\simeq A.
\]
Let
$
\delta_A\colon A \xrightarrow{\simeq} \mathbb{D}^2(A)
=\mathbb{D}(X)
$
be the biduality isomorphism. For
$f\in \operatorname{Hom}_{D(A)}(A,X)$, define its transpose by
\[
f^\dagger\colon
A \xrightarrow{\delta_A} \mathbb{D}(X)
\xrightarrow{\mathbb{D}(f)} \mathbb{D}(A)=X.
\]
Thus we obtain a map
\[
\tau\colon
\operatorname{Hom}_{D(A)}(A,X)
\longrightarrow
\operatorname{Hom}_{D(A)}(A,X),
\qquad
f\longmapsto f^\dagger.
\]
Since $\mathbb{D}$ is a contravariant equivalence,  $\tau$ is a bijection. Moreover, $\tau$ is $A$-linear.
Thus, under the natural isomorphism
$
\operatorname{Hom}_{D(A)}(A,X)
\cong H^0(X)=C,
$
the operation $f\mapsto f^\dagger$ induces an $A$-module automorphism
of $C$.
Therefore, \(f^\dagger\) corresponds to another generator of \(C\), hence
$
H^0(f^\dagger)\colon A\to C
$
is also surjective. By the same argument as above, \(\operatorname{Cone}(f^\dagger)\) has cohomology concentrated in degree \(-1\). Hence
\[
\operatorname{Cone}(f^\dagger)\simeq N[-1],
\qquad
N:=H^{-1} \bigl(\operatorname{Cone}(f^\dagger)\bigr).
\]
Applying \(\mathbb{D}\) to the exact triangle \eqref{exact_triangle} gives the following  exact triangle after rotations 
\[
A\xrightarrow{f^\dagger}X\longrightarrow \mathbb{D}(M)\longrightarrow A[-1].
\]
Therefore,
\[
\mathbb{D}(M)\simeq \operatorname{Cone}(f^\dagger)\simeq N[-1],
\]
and hence
\[
\mathbf{R}\!\operatorname{Hom}_A(M,D_A)\simeq N[-d].
\]
By local duality \cite[Thm.~V.6.2]{duality}, we have
\[
H_{\mathfrak m}^i(M)
\cong
H^{-i}\!\left(\mathbf{R}\!\operatorname{Hom}_A(M,D_A)\right)^\vee
\cong
\begin{cases}
	N^\vee, & i=d,\\
	0, & i\neq d.
\end{cases}
\]  It follows that
  \(M\) is a  maximal Cohen--Macaulay \(A\)-module.
\end{proof}

\begin{remark}
		The deficiency module \(K^{d-1}(A)\) is also related to the non-\(S_2\)
		locus.
	In the setup of Theorem~\ref{main}, suppose additionally that \(A\) is
	equidimensional and unmixed and that its non-\(S_2\) locus is nonempty.
	Set
	$
	B:=\operatorname{End}_A(\omega_A)$,
	$
	C:=B/A.
	$
	Then \(\dim C=d-2\)  as in the proof of
	\cite[Thm.~4.7]{xie_connectedness}. And by
	\cite[Thm.~4.3]{xie_connectedness} the connecting morphism
	$
	K^{d-1}(A)\to K^{d-2}(C)
	$
	is the \(S_2\)-hull map.     If
	\(K^{d-1}(A)\) is moreover equidimensional and \(S_2\), then this map is an isomorphism. Writing \(K^{d-1}(A)\cong A/I\), we have \(\dim A/I=d-2\) and \(A/I\) defines the
	top-dimensional part of the non-\(S_2\) locus which is connected in codimension one by
	\cite[Thm.~3.2]{xie_connectedness}.
\end{remark}
\subsection{Identification of the cone construction}

We now give a more explicit description of the cone module constructed above. We begin with the following lemma, which identifies the cone module with the kernel of a natural map.

\begin{lemma}\label{cone_identification}
	Let \(A\) be a commutative ring, and let \(X\in D(A)\) with
\[
H^{-1}(X)=U,\qquad
H^0(X)=V,\qquad
H^i(X)=0 \quad \text{for } i\notin\{-1,0\},
\]
where \(U\) and \(V\) are finitely generated \(A\)-modules. Let   
$
	\alpha\colon F\twoheadrightarrow V
$
	be a surjection with \(F\)  a
	projective \(A\)-module. Suppose that \(x\in A\) is a nonzerodivisor on both \(A\)
	and \(U\), and that \(xV=0\). Set
	\[
	B:=A/xA,
	\qquad
	W:=H^0\!\left(\mathbf{R}\!\operatorname{Hom}_A(B,X)\right).
	\]
	Since \(F\) is projective, there is a natural isomorphism
	$
	\operatorname{Hom}_{D(A)}(F,X)
	\cong
	\operatorname{Hom}_A(F,V).
$
	Let
$
	f\colon F\to X
$
	be the morphism in \(D(A)\) corresponding to \(\alpha\), and set
	$
	M_f:=H^{-1}\!\left(\operatorname{Cone}(f)\right).
$
	Then the following hold.
	
	\begin{enumerate}
		\item There is a canonical short exact sequence
		\[
		0\longrightarrow U/xU
		\xlongrightarrow{\overline{\tau}}
		W
		\xlongrightarrow{\pi}
		V
		\longrightarrow 0.
		\]
		
		\item There is  a lift
		$
		\widetilde{\alpha}\colon F\to W
		$
		of \(\alpha\),  so that
		$
		\pi\circ\widetilde{\alpha}=\alpha 
		$ and the following holds. 
		Define
		\[
		\Phi\colon U\oplus F\longrightarrow W,
		\qquad
		\Phi(u,z)
		=
		\overline{\tau}(\overline{u})+\widetilde{\alpha}(z),
		\]
		where \(\overline{u}\) denotes the image of \(u\) in \(U/xU\). Then \(\Phi\)
		is surjective, and 
		$
		M_f\cong\ker(\Phi).
		$
		Moreover, \(M_f\) fits into a short exact sequence
		\[
		0\longrightarrow U
		\longrightarrow M_f
		\longrightarrow \ker(\alpha)
		\longrightarrow 0.
		\]

	\end{enumerate}
\end{lemma}
\begin{proof}
By taking a good truncation of \(X\), see \cite[1.2.7, p.~9]{weibel},
we may represent \(X\) by a two-term complex
$
X=\bigl[P\xrightarrow{\delta}Q\bigr],
$
where \(P\) is placed in degree \(-1\) and \(Q\) in degree \(0\). Thus
\[
U=\ker(\delta),
\qquad
V=\operatorname{coker}(\delta).
\] 
	Since \(F\) is projective, we may lift
$
\alpha\colon F\to V
$
to a map
$
q\colon F\to Q 
$
which represents the derived morphism
\(f\colon F\to X\) corresponding to \(\alpha\). Hence
\begin{equation}\label{M_f}
M_f:=H^{-1}\bigl(\operatorname{Cone}(f)\bigr)
=
\left\{
(p,z)\in P\oplus F
\;\middle|\;
\delta(p)-q(z)=0
\right\}.
\end{equation}

Moreover, since \(\alpha\) is surjective and
$
V=\operatorname{coker}(\delta)=Q/\operatorname{im}(\delta),
$
we have $ Q=\operatorname{im}( q)+\operatorname{im}(\delta)$, and hence
$
H^0\bigl(\operatorname{Cone}(f)\bigr)=0.
$
Thus
$
\operatorname{Cone}(f)\simeq M_f[-1].
$
Projection onto the second factor \(F\)  induces a surjection \(M_f\twoheadrightarrow \ker(\alpha) \) whose kernel is naturally identified with $U=\ker \delta$. Hence, we have an exact sequence
\[
0\longrightarrow U
\longrightarrow M_f
\longrightarrow \ker(\alpha)
\longrightarrow 0.
\]

Since \(xV=0\), we have \(xq(F)\subseteq \operatorname{im}(\delta)\). Since \(F\) is projective, the map \(xq\colon F\to \operatorname{im}(\delta)\) lifts through the surjection \(\delta\colon P\twoheadrightarrow \operatorname{im}(\delta)\). Thus there exists a map \(t\colon F\to P\) such that \begin{equation}\label{delta_equality}
	\delta\circ t=xq. 
\end{equation}
Indeed, this \(t\) is precisely a chain homotopy from \(xf\) to \(0\). 
Applying the contravariant functor 
\(\mathbf{R} \operatorname{Hom}_A(-,X)\) 
to the exact triangle
$
A \xrightarrow{x} A \to B \to A[-1]
$
yields the exact triangle
\begin{equation}\label{derived_triangle}
\mathbf{R}\!\operatorname{Hom}_A(B,X)
\to
\mathbf{R}\!\operatorname{Hom}_A(A,X)
\xrightarrow{x}
\mathbf{R}\!\operatorname{Hom}_A(A,X)
\to 
\mathbf{R}\!\operatorname{Hom}_A(B,X)[-1].
\end{equation}
Since \(\mathbf{R}\!\operatorname{Hom}_A(A,X)\simeq X\), we obtain
\[
\operatorname{Cone}(x\colon X\to X)
\simeq
\mathbf{R}\!\operatorname{Hom}_A(B,X)[-1].
\]
The cone of \(x\colon X\to X\) is represented by the complex
\[
0\longrightarrow P
\xrightarrow{d^{-2}}
Q\oplus P
\xrightarrow{d^{-1}}
Q
\longrightarrow 0,
\]
where
\[
d^{-2}(r)=(-\delta r,-xr)
\qquad\text{and}\qquad
d^{-1}(v,s)=-xv+\delta s.
\]
Therefore,
\[
W
=
H^0\mathbf{R}\!\operatorname{Hom}_A(B,X)
=
H^{-1}\operatorname{Cone}(x\colon X\to X)
\cong
\frac{\{(s,v)\in P\oplus Q:\delta (s)-xv=0\}}
{\{(xr,\delta r):r\in P\}}.
\]
The long exact sequence of cohomology associated to the exact triangle
\eqref{derived_triangle} contains
\[
H^{-1}(X)\xrightarrow{x}H^{-1}(X)
\xrightarrow{\tau}
H^0\mathbf{R}\!\operatorname{Hom}_A(B,X)
\xrightarrow{\pi}
H^0(X)\xrightarrow{x}H^0(X).
\]
This yields the short exact sequence
\begin{equation}\label{exact_first}
	0\longrightarrow U/xU
	\xlongrightarrow{\overline{\tau}}
	W
	\xlongrightarrow{\pi}
	V
	\longrightarrow  0.
\end{equation}
Under the above description of \(W\), the maps \(\tau\) and \(\pi\) are given by
\[
\tau(u)=[(u,0)]
\qquad\text{and}\qquad
\pi([(s,v)])=\overline{v},
\]
where \(\overline{v}\) denotes the image of \(v\in Q\) in
\(V=Q/\operatorname{im}\delta\). Indeed, under the identification
$\mathbf{R}\!\operatorname{Hom}_A(B,X)\simeq
\operatorname{Cone}(x\colon X\to X)[-1]$, these maps are induced by the
canonical inclusion and projection in the short exact sequence of complexes
\[
0\longrightarrow X
\longrightarrow \operatorname{Cone}(x\colon X\to X)
\longrightarrow X[-1]
\longrightarrow 0,
\]
see \cite[1.5.2, p.~19 and p.~371]{weibel}.

For \(z\in F\), by \eqref{delta_equality}, we have
\(\delta(t(z))=xq(z)\). Hence \((t(z),q(z))\) is a \((-1)\)-cycle in
\(\operatorname{Cone}(x\colon X\to X)\). Define
\[
\widetilde{\alpha}\colon F\longrightarrow W,
\qquad
\widetilde{\alpha}(z)=[(t(z),q(z))].
\]
Then
$
\pi\bigl(\widetilde{\alpha}(z)\bigr)
=[q(z)]
=\alpha(z),
$
so \(\widetilde{\alpha}\) is indeed a lift of \(\alpha\).
Then 
$
\Phi\colon U\oplus F\longrightarrow W$ ,
$\Phi(u,z)
=
\overline{\tau}( {\overline{u}})+\widetilde{\alpha}(z)$ is given by $\Phi(u,z) =[(u+t(z), q(z))]$. Using the exact sequence \eqref{exact_first} and   that
$\alpha=\pi\circ\widetilde{\alpha}$ is surjective, it follows that
$\Phi$ is surjective.

Define the homomorphism
\[
\Theta\colon M_f\longrightarrow U\oplus F,
\qquad
\Theta(p,z)=(xp-t(z),z).
\]
First, \(\Theta\) is well defined. Indeed, for \((p,z)\in M_f\), by
\eqref{M_f} and \eqref{delta_equality}, we have
\[
\delta(xp-t(z))
=x\delta(p)-\delta(t(z))
=xq(z)-xq(z)
=0.
\]
Hence \(xp-t(z)\in\ker\delta=U\).
Next,
\[
\Phi\circ\Theta(p,z)
=
\Phi(xp-t(z),z)
=
[(xp,q(z))]
=
[(xp,\delta(p))].
\]
Since \((xp,\delta(p))\) is a boundary in
\(\operatorname{Cone}(x\colon X\to X)\), it follows that
\(\Phi\circ\Theta=0\). Thus
$
\operatorname{im}\Theta\subseteq\ker\Phi.
$

We next prove that \(\Theta\) is injective. Suppose that
\(\Theta(p,z)=0\). Then \(z=0\) and \(xp=0\). Since
\((p,0)\in M_f\), \eqref{M_f} gives \(\delta(p)=0\), so \(p\in U\).
As \(x\) is regular on \(U\), we have \(p=0\).

Finally, we show that \(\Theta\) is surjective onto \(\ker\Phi\).
Let \((u,z)\in\ker\Phi\). Then
$
[(u+t(z),q(z))]=0
$
in \(W\). Hence there exists \(p\in P\) such that
$
u+t(z)=xp$,
$
q(z)=\delta(p).
$
It follows that \((p,z)\in M_f\), and
$
\Theta(p,z)
=
(xp-t(z),z)
=
(u,z).
$

Therefore,
$
\Theta\colon M_f\xrightarrow{\sim }\ker\Phi
$
is an isomorphism. This finishes the proof. 
\end{proof}

\begin{remark}
	If \(x\) is not regular on \(U\), the same construction yields an exact sequence
	\[
	0\longrightarrow (0:_U x)
	\longrightarrow M_f
	\xrightarrow{\Theta}
	\ker\Phi
	\longrightarrow 0.
	\]
\end{remark}
In \cite{tavanfar_Shimomoto}, Tavanfar and Shimomoto proved the following case of the
small Cohen--Macaulay conjecture.
\begin{theorem}[{\cite[Thm.~3.2]{tavanfar_Shimomoto}}]\label{tanvanfar}
	Let \((A,\mathfrak m)\) be a  quasi-Gorenstein local ring
	which is a homomorphic image of a
	Gorenstein local ring. Suppose that the following condition holds:
	\begin{equation}
		\tag{*}\label{condition_TS}
		\parbox{0.9\textwidth}{
			There exists a regular sequence \(\mathbf y\) such that
			\(A/\mathbf yA\) is a \(3\)-dimensional quasi-Gorenstein  
			local ring satisfying
			\( H^2_{\mathfrak m}(A/\mathbf yA )=k\).
		}
	\end{equation}
	Then there exists some
	$
	x\in \operatorname{ann}_A K^{d-1}(A),
	$ such that $x,\mathbf{y}$ form a regular sequence of $A$ and
	the first syzygy \(\Omega^1\omega_{A/xA}\) of the canonical module $\omega_{A/xA}$  of
	\(A/xA\) is a maximal Cohen--Macaulay $A$-module of rank 2.
\end{theorem}
Their construction is obtained by lifting the maximal Cohen--Macaulay module constructed over the 3-dimensional quotient \(A/\mathbf yA\). More precisely, one shows that the deficiency module \(K^{d-1}(A)\), the canonical module \(\omega_{A/xA}\), and the syzygy \(\Omega^1\omega_{A/xA}\) are compatible with reduction modulo \(\mathbf y\), and hence arise as lifts of the corresponding modules over \(A/\mathbf yA\); see \cite[p.~4]{xie} or \cite[p.~677-680]{tavanfar_Shimomoto}. In particular, \(\Omega^1\omega_{A/xA}\) is obtained by lifting the maximal Cohen--Macaulay module \(\Omega^1\omega_{A/(x,\mathbf y)A}\) from the 3-dimensional quotient. Thus the construction in their original proof   reduces to the following 3-dimensional situation.
\begin{theorem}[{\cite[Thm.~3.2, reduction]{tavanfar_Shimomoto}}]\label{reduction}
	Let \((A,\mathfrak m)\) be a 3-dimensional  quasi-Gorenstein local ring 
	which is a homomorphic image of a Gorenstein local ring. Suppose that
	\(\depth A=2\) and \(H^2_{\mathfrak m}(A)\cong k\). Then, for any
	nonzerodivisor \(x\in \m\), the first syzygy module
	$
	\Omega^1\omega_{A/xA}
	$
	of the canonical module \(\omega_{A/xA}\) of \(A/xA\) is a   maximal Cohen--Macaulay $A$-module.
\end{theorem}

 In \cite{xie}, we gave a simpler proof of both the reduction steps and the 3-dimensional case.

Using Lemma \ref{cone_identification}, we can extend this result and show
that the cone construction in
Theorem \ref{main} is naturally identified with the first syzygy module
$
\Omega^1\omega_{A/xA} 
$
whenever \(A\) is quasi-Gorenstein. 

\begin{proposition}\label{cor_syzygy}
	Let \((A,\mathfrak m)\) be a quasi-Gorenstein local ring of dimension $d$ which is a
	homomorphic image of a Gorenstein local ring. Suppose that
	$
	\depth A=d-1
	$
	and that \(K^{d-1}(A)\) is cyclic. Let \(x\in \operatorname{ann}_A K^{d-1}(A)\)
	be a nonzerodivisor on $A$. Then the first syzygy 
	$
	\Omega^1\omega_{A/xA}
	$
	is a  maximal Cohen--Macaulay \(A\)-module.
\end{proposition}
\begin{proof}
	Since \(A\) is quasi-Gorenstein, we have \(\omega_A\cong A\). The complex
	\(X\) appearing in the proof of Theorem \ref{main} then satisfies
$
	H^{-1}(X)=\omega_A\cong A$,
 $
	H^0(X)=K^{d-1}(A),
$
and $H^i(X)=0$, $i\notin\{-1,0\}$.
	Moreover,
	$
	H^0\mathbf{R}\!\operatorname{Hom}_A(B,X)\cong \omega_B,
	$
	where \(B=A/xA\). Hence, by Lemma \ref{cone_identification}, there is a
	short exact sequence
\begin{equation}\label{cor_exact}
		0\longrightarrow A/xA
	\xrightarrow{\overline{\tau}}
	\omega_B
	\longrightarrow
	K^{d-1}(A)
	\longrightarrow 0.
\end{equation}
	
	Let \(\alpha\colon A\to K^{d-1}(A)\) be a surjection, and let
	\(\widetilde{\alpha}\colon A\to\omega_B\) be the lift given by
	Lemma \ref{cone_identification}. Then the map
	$
	\Phi\colon A^2\to\omega_B$,
$
	\Phi(u,z)
	=
	\overline{\tau}(\overline{u})
	+\widetilde{\alpha}(z),
	$
	is surjective, and Lemma \ref{cone_identification} gives
	$
	M_f:= H^{-1}\bigl(\operatorname{Cone}(f)\bigr)\cong\ker\Phi.
	$
	By Theorem \ref{main}, \(M_f\) is   
	maximal Cohen--Macaulay.
	
It remains to show that \(M_f\) is the first syzygy
module of \(\omega_B\). It suffices to show that \(\omega_B\) is minimally generated by two
elements. Suppose otherwise that \(\omega_B\) is generated by one
element, and hence is cyclic.  Since
$B\hookrightarrow \omega_B$, it follows that $\omega_B\cong B$.
Hence \eqref{cor_exact} gives 
 $K^{d-1}(A)\cong B/aB$, for some nonzerodivisor $a\in B$, and therefore
$
\dim K^{d-1}(A)=\dim B-1=d-2.
$
On the other hand, since \(A\cong\omega_A\) is
\(S_2\),   by \cite[Prop.~3.1]{canonical_module},
$
\dim K^{d-1}(A)\leq d-3,
$
a contradiction. Therefore, \(\omega_B\) is minimally generated by two
elements  and  
$
M_f
\cong
\ker\Phi
\cong
\Omega^1\omega_B.
$
\end{proof}
 
As pointed out by the referee, Theorem~\ref{tanvanfar} of Tavanfar and
Shimomoto can be reformulated in terms of the condition
\(e_{\mathfrak m}(K^{d-1}(A))=1\), where \(e_{\mathfrak q}\) denotes
the Hilbert--Samuel multiplicity with respect to \(\mathfrak q\).
We now show that this multiplicity condition is equivalent to the
deformation condition~\eqref{condition_TS}.
	\begin{proposition}\label{equivalent}
		For  \((A,\mathfrak m)\)   a  quasi-Gorenstein local ring of dimension $d$
		which is a homomorphic image of a
		Gorenstein local ring, the condition \eqref{condition_TS} in Theorem \ref{tanvanfar} is equivalent to   $e_{\mathfrak{m} }(K^{d-1}(A))  =1$ and $\depth A =d -1$.
	\end{proposition}
	\begin{proof}
		Note that for $A$   quasi-Gorenstein that is a homomorphic image of a Gorenstein local ring with $\depth A =d-1$,  \(K^{d-1}(A)\) is Cohen--Macaulay
		of dimension \(d-3\); see \cite[p.~2]{xie}.
		
		Suppose first that \eqref{condition_TS} holds. Applications of the local cohomology long exact sequence along \(\mathbf y\) give the depth condition on \(A\) and
		\(K^{d-1}(A)/\mathbf yK^{d-1}(A)\cong K^2(A/\mathbf yA)\cong k\).
		Since \(K^{d-1}(A)\) is Cohen--Macaulay and \(\mathbf y\) is a system
		of parameters for \(K^{d-1}(A)\), we have
		\(e_{(\mathbf y)}(K^{d-1}(A))
		=\ell(K^{d-1}(A)/\mathbf yK^{d-1}(A))=1\) by \cite[Thm.~IV.3]{serre_local_algebra}. Hence
		\( e_{(\mathbf y)}(K^{d-1}(A))
		\geq e_{\mathfrak m}(K^{d-1}(A)) \) implies
		\(e_{\mathfrak m}(K^{d-1}(A))=1\). 
		
		Conversely, suppose that $\depth A =d-1$ and 
		\(e_{\mathfrak m}(K^{d-1}(A))=1\). Since \(K^{d-1}(A)\) is
		Cohen--Macaulay, we have
		\(\mu_A(K^{d-1}(A))\leq e_{\mathfrak m}(K^{d-1}(A))=1\), so
		\(K^{d-1}(A)\) is cyclic. Write
		\(K^{d-1}(A)\cong A/I=:\overline A\), and let
		\(\overline{\mathfrak m}:=\mathfrak m/I\). Then
		\(e_{\overline{\mathfrak m}}(\overline A)=1\).
		
		We claim that \(\overline A\) is regular. If the residue field is
		infinite, choose a   reduction
		\(J=(\overline x_1,\ldots,\overline x_{d-3})\) of
		\(\overline{\mathfrak m}\), generated by a system of parameters of
		\(\overline A\), as in \cite[Thm.~14.14]{matsumura}. Since
		\(\overline A\) is Cohen--Macaulay,
		\(\ell(\overline A/J)=e_J(\overline A)
		=e_{\overline{\mathfrak m}}(\overline A)=1\). Thus
		\(J=\overline{\mathfrak m}\), and hence \(\overline A\) is regular.
		If the residue field is finite, we pass to the faithfully flat local
		extension
		\(\overline A[X]_{\overline{\mathfrak m}\overline A[X]}\), which has
		infinite residue field, and apply the same argument to show that \(\overline A[X]_{\overline{\mathfrak m}\overline A[X]}\) is regular; see
		\cite[pp.~114--115]{matsumura}. Hence
		\(\overline A\) is regular.
		
		Now choose
		\(\mathbf y=y_1,\dots,y_{d-3}\in\mathfrak m\) such that
		\(y_1,\dots,y_{d-3}\) is an \(A\)-regular sequence and its image in
		\(\overline A\) is a regular system of parameters. 
	This can be done inductively by prime avoidance. Let
	\(A_i=A/(y_1,\ldots,y_i)\), and let \(I_i\) and \(\mathfrak m_i\)
	denote the images of \(I\) and \(\mathfrak m\) in \(A_i\)  respectively. One chooses
	$
	y_{i+1}$ not in  $
	\bigcup_{\mathfrak p\in\operatorname{Ass}(A_i)}\mathfrak p
	$ and not in $(\mathfrak m_i^2+I_i).
$
		Then $A/\mathbf y  A$ is quasi-Gorenstein and 
		\(K^2(A/\mathbf yA)\cong
		K^{d-1}(A)/\mathbf yK^{d-1}(A)
		\cong\overline A/\mathbf y\overline A\cong k\).  
	Hence  \(\mathbf y\) satisfies
		\eqref{condition_TS}.
	\end{proof}
Therefore, Theorem \ref{tanvanfar} has the following equivalent 
reformulation. 
\begin{theorem}[{\cite[Thm.~3.2, Reformulation]{tavanfar_Shimomoto}}]\label{reformulation}
	If \((A,\mathfrak{m})\) is a quasi-Gorenstein local ring that is a
	homomorphic image of a Gorenstein local ring, with
	\(\depth A=\dim A-1\) and
	\(e_{\mathfrak m}(K^{d-1}(A))=1\), then  for a suitable nonzerodivisor
	\(x\in\operatorname{ann}_A K^{d-1}(A)\),
	\(\Omega^1\omega_{A/xA}\) is  maximal Cohen--Macaulay.
\end{theorem}

As mentioned in the proof of Proposition~\ref{equivalent}, the
multiplicity 1 condition \(e_{\mathfrak m}(K^{d-1}(A))=1\) implies that \(K^{d-1}(A)\) is cyclic.
Thus Proposition~\ref{cor_syzygy} applies both to the reduction in Theorem~\ref{reduction} and to Theorem~\ref{tanvanfar} (equivalently, Theorem~\ref{reformulation}). In either case, we recover and extend the construction: \(\Omega^1\omega_{A/xA}\) is maximal Cohen--Macaulay for any nonzerodivisor \(x\in\operatorname{ann}_A K^{d-1}(A)\), independently of any choice of deformation.

\subsection{Higher syzygies of lower-dimensional Cohen--Macaulay modules}\label{higher_syzygy}

In the setup of Proposition \ref{cor_syzygy}, the exact sequence
\[
0\longrightarrow M_f\longrightarrow A^2\longrightarrow \omega_B\longrightarrow 0
\]
gives
$
\depth_A\omega_B
\geq
\min\{\depth_A M_f-1,\depth A\}
\geq d-1.
$
Hence \(\omega_B\) is a Cohen--Macaulay \(A\)-module of dimension \(d-1\).

There are many natural sources of Cohen--Macaulay modules of lower
dimension. For example, in the introduction of our previous work \cite{xie},
we recorded the fact that if \((A,\mathfrak m)\) is quasi-Gorenstein and is
a homomorphic image of a Gorenstein local ring of dimension \(d\), with
\(\depth A=d-1\), then \(K^{d-1}(A)\) is Cohen--Macaulay of dimension
\(d-3\). Another example is obtained by taking a quotient of \(A\) of
dimension \(2\); its \(S_2\)-ification is then a Cohen--Macaulay module of
dimension \(2\). 

The condition below isolates the obstruction to obtaining a maximal Cohen--Macaulay module from a higher syzygy.
\begin{proposition}\label{prop_syzygy}
	Let \(A\) be a local ring of dimension \(d\) which is a
	homomorphic image of a Gorenstein local ring. Suppose \(\depth A=d-1\),
	and let \(L\) be a finitely generated Cohen--Macaulay \(A\)-module of
	dimension \(d-r\), \(r\ge 1\). Then
$
	\depth \Omega^i L\geq d-1
	$
	for every \(i\geq r-1\). Moreover, \(\Omega^r L\) is maximal
	Cohen--Macaulay if and only if the natural map
	\[K^{d-1}\bigl(\Omega^{r-1}L\bigr)
	\to
	K^{d-1}(A)^{\beta_{r-1}}\]
	is surjective, where \(\beta_i \) denotes the \(i\)-th
	Betti number of \(L\).
\end{proposition}
\begin{proof}
Denote by \(Z_i:=\Omega^iL\) the \(i\)-th syzygy of \(L\). Let
\(F_\bullet\) be a minimal free resolution of \(L\), with
\(F_i=A^{\beta_i}\). For each \(i\geq 0\), we have a short exact sequence
\begin{equation}\label{exact_syzygy}
	0\longrightarrow Z_{i+1}
	\longrightarrow F_i
	\longrightarrow Z_i
	\longrightarrow 0.
\end{equation}
Then,
$
\depth Z_{i+1}
\geq
\min\{\depth F_i,\depth Z_i+1\}
=
\min\{d-1,\depth Z_i+1\}.
$
Since
$
\depth Z_0=\depth L=d-r,
$
 we have 
$
\depth Z_i\geq d-1$ 
 for all $i\geq r-1$.
 
The long exact sequence of local cohomology associated to
\eqref{exact_syzygy} with \(i=r-1\) contains
\[
0
\longrightarrow
H_{\mathfrak m}^{d-1}(Z_r)
\longrightarrow
H_{\mathfrak m}^{d-1}(F_{r-1})
\longrightarrow
H_{\mathfrak m}^{d-1}(Z_{r-1}).
\]
So \(Z_r\) is maximal
Cohen--Macaulay if and only if
  the natural map
$
H_{\mathfrak m}^{d-1}(F_{r-1})
\to
H_{\mathfrak m}^{d-1}(Z_{r-1})
$
is injective. By local duality and faithful flatness of completion, this is equivalent to the surjectivity of
%$
%\widehat{K^{d-1}(Z_{r-1})}
%\to
%\widehat{K^{d-1}(F_{r-1})}.
%$
%which is equivalent to the surjectivity of
$
K^{d-1}(Z_{r-1})
\to
K^{d-1}(F_{r-1})
=
K^{d-1}(A)^{\beta_{r-1}} 
$. 
\end{proof}

\begin{corollary}
In the setup of Theorem \ref{reduction}, the first syzygy 
	\(\Omega^1\omega_{A/xA}\) is maximal Cohen--Macaulay.
\end{corollary}

\begin{proof}
	The \(A\)-module \(\omega_{A/xA}\) is Cohen--Macaulay of dimension \(2\).
	By Proposition~\ref{prop_syzygy}, it is enough to show that the natural
	map
$
	K^2(\omega_{A/xA})\to K^2(A)^2
$ is surjective. This is precisely the map
	\[
	f\colon
	\operatorname{Ext}_R^1(\omega_{A/xA},R)
	\longrightarrow
	\operatorname{Ext}_R^1(A^2,R)
	\cong k^2,
	\]
	which was shown to be surjective in \cite[p.~5]{xie}.	
\end{proof}

\section{Examples of Section Rings}

In this section, we give some examples of section rings to which our main
theorem applies.  

\begin{proposition}\label{prop_example}
	Let \(k\) be a field, and let \(r,m\geq 1\). Let \(Z\) be a projective
	\(k\)-variety of dimension \(r\) satisfying
	\begin{equation}\label{condition_a}
			H^i(Z,\mathcal O_Z)\cong
		\begin{cases}
			k, & i=0,r,\\
			0, & \text{otherwise}.
		\end{cases}
	\end{equation}
	Let \(S\) be a projective \(k\)-variety of dimension \(m\) satisfying
	\begin{equation}\label{condition_b}
			H^j(S,\mathcal O_S)\cong
		\begin{cases}
			k, & j=0,\\
			0, & j>0.
		\end{cases}
	\end{equation}
	Set
	$
	X_0=Z\times_k S.
	$
	For \(1\leq j\leq m-1\), suppose that \(X_j\subset X_{j-1}\) is an
	effective Cartier divisor and that
	\begin{equation}\label{condition_c}
		H^q\bigl(X_{j-1},
		\mathcal O_{X_{j-1}}(-X_j)\bigr)=0
		\qquad
		\text{for all }q<\dim X_{j-1}.
	\end{equation}
	If \(m=1\), set \(X=X_0\). If \(m\geq 2\), set
	$
	X=X_{m-1}.
	$
	Assume that \(X\) is integral.
	Let \(L\) be an ample line bundle on \(X\) satisfying
\begin{equation}\label{condition_d}
		H^i(X,L^n)=0
	\qquad
	\text{for all }n\neq 0
	\text{ and }1\leq i\leq r.
\end{equation}
	Define the section ring
	\[
	R=R(X,L)
	=
	\bigoplus_{n\geq 0}H^0(X,L^n),
	\]
	let \(\mathfrak m=R_+\), and set
	$
	A=R_{\mathfrak m}.
	$
	Then
	$
	\dim A=r+2$, 
	$\depth A=r+1,
	$
	and
	$
	H^{r+1}_{\mathfrak m A}(A)\cong k.
	$
\end{proposition}
\begin{proof}
	First, we have $
	\dim X_0
	=
	\dim(Z\times_k S)
	=
	r+m$,
	and 
$
	\dim X
	=
	\dim X_0-(m-1)
 =
	r+1.
	$

	By the K\"unneth formula,
	\[
	H^i(X_0,\mathcal O_{X_0})
	\cong
	\bigoplus_{a+b=i}
	H^a(Z,\mathcal O_Z)\otimes_k H^b(S,\mathcal O_S).
	\]
	Thus by \eqref{condition_a} and \eqref{condition_b}, we have\[
	H^i(X_0,\mathcal O_{X_0})\cong 
	H^i(Z,\mathcal O_Z)\cong
	\begin{cases}
	k, & i=0,r,\\
	0, & \text{otherwise}.
\end{cases}
	\] 
   For each \(j\), there is a short exact sequence
   \[
   0\longrightarrow
   \mathcal O_{X_{j-1}}(-X_j)
   \longrightarrow
   \mathcal O_{X_{j-1}}
   \longrightarrow
   \mathcal O_{X_j}
   \longrightarrow 0.
   \]
   The associated long exact sequence in cohomology contains
   \[
   H^i\bigl(X_{j-1},\mathcal O_{X_{j-1}}(-X_j)\bigr)
   \longrightarrow
   H^i\bigl(X_{j-1},\mathcal O_{X_{j-1}}\bigr)
   \longrightarrow
   H^i\bigl(X_j,\mathcal O_{X_j}\bigr)
   \longrightarrow
   H^{i+1}\bigl(X_{j-1},\mathcal O_{X_{j-1}}(-X_j)\bigr).
   \]  
  By \eqref{condition_c},
  $
  H^i\bigl(X_{j-1},\mathcal O_{X_{j-1}}\bigr)
  \cong
  H^i\bigl(X_j,\mathcal O_{X_j}\bigr)
  $
  for every \(i\leq \dim X_{j-1}-2\). Since \(j\leq m-1\), we have
  $
  \dim X_{j-1}\geq r+2.
 $
  Therefore, iterating these isomorphisms gives
  \begin{equation}\label{condition_e}
  	 H^i(X,\mathcal O_X)
  	\cong
  	H^i(X_0,\mathcal O_{X_0})
  	\qquad
  	\text{for }0\leq i\leq r.
  \end{equation}
Since \(X\) is integral and \(L\) is ample, we have
\(H^0(X,L^j)=0\) for \(j<0\). Moreover, the section ring  
\[
R=R(X,L):=\bigoplus_{n\geq 0}H^0(X,L^n)
\]
 is a
finitely generated graded domain over \(k\), with
\(\operatorname{Proj}R\cong X\) and
$
\widetilde{R(n)}\cong L^n 
$, $n\in \mathbb{Z}$, see \cite[\href{https://stacks.math.columbia.edu/tag/0B5T}{Tag 0B5T}]{stacks-project} and \cite[Thm.~4.5.2, Prop.~4.5.5]{ega2}.
Hence
$
\dim A= \dim R=\dim X+1=r+2.
$
By the Serre--Grothendieck correspondence
\cite[Thm.~20.4.4]{local_cohomology}, for every \(i\geq 1\),
\[
H_{R_+}^{i+1}(R)
\cong
\bigoplus_{j\in\mathbb Z} H^i(X,L^j).
\]
By \eqref{condition_d} and \eqref{condition_e},
$
H_{R_+}^{i+1}(R)=0$
  for $1\leq i\leq r-1$,
and
$
H_{R_+}^{r+1}(R)\cong k.
$   
In the exact sequence \cite[Thm.~2.2.6(c)]{local_cohomology}
\[
0\longrightarrow H_{R_+}^{0}(R)
\longrightarrow R
\longrightarrow \bigoplus_{j\in\mathbb Z} H^0(X,L^j)
\longrightarrow H_{R_+}^{1}(R)
\longrightarrow 0,
\]
the middle map is an isomorphism, as $H^0(X, L^j)=0$ for $j\leq -1$. Hence
$
H_{R_+}^{0}(R)=H_{R_+}^{1}(R)=0.
$
Localizing at the homogeneous maximal ideal \(\mathfrak m=R_+\) and using local duality give the desired result.
\end{proof}

The conditions of Proposition \ref{prop_example} can be constructed using
Kodaira vanishing. For example, one may take \(Z\) to be a Calabi--Yau
variety over a field $k$ of characteristic \(0\) and
\(S=\mathbb P_k^m\).
Throughout this section, we call a smooth geometrically
connected projective  variety \(Z\) over \(k\) a \emph{projective Calabi--Yau variety}
if
\[
\omega_Z\cong\mathcal O_Z
\qquad\text{and}\qquad
H^i(Z,\mathcal O_Z)=0
\quad\text{for }0<i<\dim Z.
\]
\begin{corollary}\label{cor_example}
Let \(k\) be a field of characteristic \(0\), and let \(r\geq 1\) and
\(m\geq 2\).  Let \(Z\) be a
projective  Calabi--Yau variety of dimension \(r\) over \(k\). Let \(S\) be
a smooth projective geometrically connected variety of dimension \(m\)
satisfying
\[
H^j(S,\mathcal O_S)=0
\qquad\text{for all }j>0.
\]
Let
$
X_0=Z\times S.
$
For \(1\leq j\leq m-1\), let
$
X_j\subset X_{j-1}
$
be a smooth ample effective Cartier divisor, and set
$
X=X_{m-1}.
$
Assume that \(\omega_X\) is ample. Let \(a\geq 2\), and set
$
L=\omega_X^{\otimes a}.
$
Then the local section ring
\[
A=
\Bigl(
\bigoplus_{n\geq 0}
H^0\bigl(X,\omega_X^{\otimes an}\bigr)
\Bigr)_{R_+}
\]
satisfies
\[
\dim A=r+2,
\qquad
\depth A=r+1,
\qquad
K^{r+1}(A)\cong k.
\]
\end{corollary}
\begin{proof}
	We check the conditions in Proposition \ref{prop_example} are satisfied. 
	Since \(Z\) is geometrically connected,
	$
	H^0(Z,\mathcal O_Z)\cong k
	$
	by \cite[\href{https://stacks.math.columbia.edu/tag/0FD2}{Tag 0FD2}]{stacks-project}. Since
	\(\omega_Z\cong\mathcal O_Z\), Serre duality
	\cite[Cor.~7.7]{hartshorne} gives
	$
	H^r(Z,\mathcal O_Z)'
	\cong
	H^0(Z,\omega_Z)
	\cong k.
	$
	Together with the assumed vanishing
$
	H^i(Z,\mathcal O_Z)=0 $ for $0<i<r$
 shows that \(Z\) satisfies \eqref{condition_a}. Similarly, \(S\)
	satisfies \eqref{condition_b}.
	
	For each \(j\), the line bundle
	$
	\mathcal O_{X_{j-1}}(X_j)
	$
	is ample. By Serre duality,
	\[
	H^q\bigl(X_{j-1},\mathcal O_{X_{j-1}}(-X_j)\bigr) {'}
	\cong
	H^{\dim X_{j-1}-q}
	\bigl(
	X_{j-1},
	\omega_{X_{j-1}}\otimes
	\mathcal O_{X_{j-1}}(X_j)
	\bigr).
	\]
	By Kodaira vanishing in characteristic \(0\), see
	\cite[Cor.~2.11]{deligne_illusie}, the latter  vanishes whenever
	\(q<\dim X_{j-1}\). Hence \eqref{condition_c} is satisfied.
	
	Now let \(n>0\). Since
	$
	L^n=\omega_X^{\otimes an}
	=
	\omega_X\otimes\omega_X^{\otimes(an-1)}
	$
	and \(a\geq2\), the line bundle
	\(\omega_X^{\otimes(an-1)}\) is ample. Kodaira vanishing therefore gives
	$
	H^i(X,L^n)=0$
	 {for all }$i>0.
$
	
	For \(n<0\), since \(\dim X=r+1\), by  Serre duality and Kodaira vanishing applied to the ample line bundle \(L^{-n}=\omega_X^{\otimes(-an)}\) gives
	\[
	H^i(X,L^n) {'}
	\cong
	H^{r+1-i}(X,\omega_X\otimes L^{-n})=0
	\qquad\text{whenever }r+1-i>0.
	\]
	Thus
$
	H^i(X,L^n)=0
$  for  $i<r+1, n\neq 0. 
	$
	Therefore \eqref{condition_d} is satisfied for \(L\). Finally, the argument proving \eqref{condition_e} gives
	\(H^0(X,\mathcal O_X)
	\cong k\) thus $X$ is geometrically connected by \cite[\href{https://stacks.math.columbia.edu/tag/0FD2}{Tag 0FD2}]{stacks-project}.
	Moreover, since \(X\) is smooth and geometrically connected,
	it is geometrically integral.  Therefore all the hypotheses of
	Proposition~\ref{prop_example} are satisfied.
\end{proof}
\begin{example}
	Let \(Z\) be a projective Calabi--Yau variety over a field \(k\) of characteristic
	\(0\), and let \(M\) be an ample line bundle on \(Z\). Set
	\(Y:=Z\times \mathbb P_k^2\), and write \(H:=p_1^*M\) and
	\(F:=p_2^*\mathcal O_{\mathbb P_k^2}(1)\), where \(p_1\) and \(p_2\)
	are the two projections. Choose integers \(c\gg 0\) and \(b>3\), and let
	\(X\in |cH+bF|\) be a smooth divisor. Since
	\(\omega_Y\cong p_1^*\omega_Z\otimes p_2^*\omega_{\mathbb P_k^2}
	\cong \mathcal O_Y(-3F)\), the adjunction formula gives
	$
	\omega_X\cong \mathcal O_X\bigl(cH+(b-3)F\bigr).
	$
	Since \(c>0\) and \(b-3>0\), the line bundle \(\omega_X\) is ample.
	For any \(a\geq2\), set \(L:=\omega_X^{\otimes a}\). Then \(X\) and
	\(L\) satisfy the hypotheses of Corollary \ref{cor_example}. 
\end{example}

Moreover, using the criterion of Bhatt \cite[Thm.~1.3]{bhatt}, one can
obtain classes of rings in characteristic \(p\) whose completions admit
no small Cohen--Macaulay algebra, but which satisfy
Proposition \ref{prop_example} and hence admit a maximal
Cohen--Macaulay module. Although Kodaira vanishing fails in general in
characteristic \(p\), certain varieties still satisfy the vanishing
conditions required in Proposition \ref{prop_example}.

\begin{example}
Let \(k\) be a perfect field of characteristic \(p>0\). Let \(Z\) be
either an elliptic curve over \(k\) or a \(K3\) surface of finite height,
and set \(X:=Z\times \mathbb P_k^1\). Let \(M\) be an ample line bundle
on \(Z\), and set
$
L:=p_1^*M\otimes p_2^*\mathcal O_{\mathbb P_k^1}(1).
$
Then the completion of the section ring \(R=R(X,L)\) at its homogeneous
maximal ideal admits no module finite Cohen--Macaulay algebra by
Bhatt  \cite[Thm.~1.3, Ex.~3.11]{bhatt}. Moreover,  using Kodaira vanishing for K3 surfaces in positive
characteristic \cite[Prop.~3.4]{mukai_kodaira} and Serre duality, together with an argument similar
to that of Corollary~\ref{cor_example},   \(X\) and \(L\)
satisfy the hypotheses of Proposition~\ref{prop_example}. 
Consequently, \(R(X,L)_{R_+}\) admits a maximal Cohen--Macaulay
module, and hence so does its completion.
\end{example}

\section*{Acknowledgments}
We thank the referee for insightful comments and suggestions, in particular for pointing out the reformulation with the condition \(e_{\mathfrak m}(K^{d-1}(A))=1\) which led to Proposition~\ref{equivalent}, and for pointing out the recent developments \cite{anghel,chen} about counterexamples to the small Cohen--Macaulay conjecture.

When I was a graduate student at the University of Illinois at Urbana--Champaign, Sankar Dutta suggested finding a simpler proof of \cite[Thm.~3.2]{tavanfar_Shimomoto}, which led to \cite{xie}. Dutta's way of seeing commutative algebra influenced and inspired my own viewpoint. The idea leading to the current work first developed when I was revisiting the earlier work during a summer visit to my mother. I am grateful to my mother for her continuous support and remember fondly my time at the University of Illinois at Urbana--Champaign, where I learned much from  Dutta and many others.

\end{document}